\documentclass[11pt]{amsart}
\usepackage{amsmath,amsthm,amssymb,bbm}
\usepackage{tikz}
\usepackage{epsfig}

\newtheorem{thm}[equation]{Theorem}

\newtheorem{lem}[equation]{Lemma}
\newtheorem{prop}[equation]{Proposition}
\theoremstyle{definition}

\newtheorem{rem}[equation]{Remark}

\numberwithin{equation}{section}
\newcommand{\1}{\mathbbm{1}}

\def\P{{\mathbb{P}}}
\def\R{{\mathbb{R}}}
\def\N{{\mathbb{N}}}
\def\A{{\mathbb{A}}}
\def\E{{\mathbb{E}}}

\renewcommand{\leq}{\leqslant}
\renewcommand{\ge}{\geqslant}\renewcommand{\geq}{\geqslant}

\DeclareMathOperator*{\argmin}{arg\,min}

\begin{document}

\title[Self-similar solutions of kinetic-type equations]{Self-similar solutions of kinetic-type equations:\\
the boundary case}

\author[K.~Bogus, D.~Buraczewski and A.~Marynych]
{Kamil Bogus, Dariusz Buraczewski and Alexander Marynych}

\address{Kamil Bogus, Faculty of Pure and Applied Mathematics, Wroc{\l}aw University of Science
and Technology, ul. Wybrze\.ze Wyspia\'nskiego 27, 50-370 Wroc{\l}aw, Poland}
\email{kamil.bogus@pwr.edu.pl}

\address{Dariusz Buraczewski, Mathematical Institute, University of
Wroc{\l}aw, Plac Grunwaldzki 2/4, 50-384 Wroc{\l}aw, Poland}
	\email{dbura@math.uni.wroc.pl}

\address{Alexander Marynych, Faculty of Computer Science and Cybernetics, Taras Shevchenko National University of Kyiv, 01601 Kyiv, Ukraine}
\email{marynych@unicyb.kiev.ua}

\thanks{K.~Bogus and D.~Buraczewski were partially supported by the National Science Center, Poland (Sonata Bis, grant number DEC-2014/14/E/ST1/00588). A.~Marynych was partially supported by the Return Fellowship of the Alexander von Humboldt Foundation.}

	\keywords{Biggins martingale, derivative martingale, Kac model, kinetic equation, random trees, smoothing transform}
	\subjclass[2010]{60F05, 82C40}

\begin{abstract}
For a time dependent family of probability measures $(\rho_t)_{t\ge 0}$ we consider a~kinetic-type evolution equation $\partial \phi_t/\partial t + \phi_t = \widehat{Q} \phi_t$ where $\widehat{Q}$ is a smoothing transform and $\phi_t$ is the Fourier--Stieltjes transform of $\rho_t$. Assuming that the initial measure $\rho_0$ belongs to the domain of attraction of a stable law, we describe asymptotic properties of $\rho_t$, as $t\to\infty$. We consider the boundary regime when the standard normalization leads to a degenerate limit and find an appropriate scaling ensuring a non-degenerate self-similar limit. Our approach is based on a probabilistic representation of probability measures $(\rho_t)_{t\ge 0}$ that refines the corresponding construction proposed in Bassetti and Ladelli [Ann. Appl. Probab. 22(5): 1928--1961, 2012].
\end{abstract}

\maketitle

\section{Introduction}

In the paper we consider a kinetic-type evolution equation  for a time dependent family of probability measures $(\rho_t)_{t\ge 0}$. Let
$$\phi(t,\xi)=\int_\R e^{i\xi v}\rho_t({\rm d}v),\quad t\geq 0,\quad \xi\in\R,
$$
be the Fourier--Stieltjes transform (the characteristic function) of $\rho_t$. We are interested in the solution of the following Cauchy problem
\begin{equation} \label{eqboltzivp}
\frac{\partial }{\partial t} \phi(t,\xi)+\phi(t,\xi)=\widehat{Q}(\phi(t,\cdot),\ldots,\phi(t,\cdot))(\xi),\quad t>0,\quad \phi(0,\xi)=\phi_0(\xi),\quad \xi\in\R,
\end{equation} where $\widehat{Q}$ is a smoothing transform. The smoothing transform $\widehat{Q}$ is defined by the equality
$$\widehat{Q}(\phi_1,\ldots,\phi_N)(\xi):=\E(\phi_1(A_1\xi)\cdot  \ldots \cdot \phi_N(A_N\xi)), \quad \xi\in\R,$$
where $\phi_1,\ldots,\phi_N$ are characteristic functions, $N$ is a fixed positive integer, and a random vector $\A = (A_1, \ldots , A_N)$ consists of positive real-valued random variables defined on a common probability space $(\Omega,\mathcal{F},\mathbb{P})$.
The initial condition $\phi_0$ is the characteristic function of some random variable $X_0$ defined on $(\Omega,\mathcal{F},\mathbb{P})$.

The equation of the form \eqref{eqboltzivp} with $N=2$ and $\A = (\sin\theta,\cos \theta)$, where $\theta$ is a random angle uniformly distributed on $[0,2\pi)$, was introduced and investigated by Kac \cite{Kac} as a model of behavior of a particle in a homogeneous gas. In subsequent works the Kac model was generalized in various directions including one dimensional dissipative Maxwell models for colliding molecules \cite{Pareschi}, models describing economical dynamics \cite{Matthes} and the inelastic Boltzmann equation \cite{bobylev,bobylev2}. We refer to \cite{BasettiLadelli2012,BasettiLadelliMatthes,BassettiPerversi:2013}
 for other examples and a comprehensive bibliography.

In this paper we study asymptotic behavior of the solution $\phi$ to equation \eqref{eqboltzivp} from probabilistic point of view and prove related limit theorems. This problem was recently addressed in \cite{BasettiLadelli2012} where it was shown that under mild assumptions, which we discuss later, there exists a parameter $\mu$, depending on the initial condition $\phi_0$ and the law of $\A$, such that the rescaled solution to \eqref{eqboltzivp}, namely
 \begin{equation}\label{eq:ss1}
w(t,\xi) = \phi(t, e^{-\mu t}\xi),\quad t\geq 0,\quad \xi\in\R,
 \end{equation}
converges to a nondegenerate limit as $t\to\infty$ and the limit is a fixed point of  a smoothing transform pertained to $\widehat Q$. The main goal of our paper is to present a class of solutions to~\eqref{eqboltzivp} which, after rescaling as in \eqref{eq:ss1}, converge to a degenerate limit, yet it is possible to~find a different normalization ensuring a nondegenerate limit possessing some self-similarity properties. To achieve our aims we propose a refinement of the probabilistic construction of the solution $\phi$ presented in \cite{BasettiLadelliMatthes} and express $\phi$ via a continuous-time branching random walk.

Firstly, we state assumptions concerning the initial condition $\phi_0$. We suppose, similarly as in \cite{BasettiLadelli2012} and \cite{BassettiPerversi:2013}, that
 the distribution function $F_0$ of $X_0$ satisfies one of the following hypotheses $(H_{\gamma})$ for some $\gamma\in(0,2]$:
\begin{itemize}
\item[$(H_1)$]
either
\subitem{(a)} $\int_\R|v| \,{\rm d}F_0(v) <+\infty$ \textit{and then we set}
$m_0:=\int_\R v\,{\rm d}F_0(v)$\\
\textit{or}
\subitem{(b)}
$F_0$ \textit{ satisfies the conditions}
$$
\lim_{x \to+\infty} x \bigl(1-F_0(x)\bigr) =
\lim_{x \to-\infty} |x| F_0(x) =c_0^+\in (0,\infty),
$$
and
$$
\lim_{R\to+\infty}\int_{-R}^{R}v{\rm d}F_0(v)=:m_0\in(-\infty,\infty).
$$
\item[$(H_2)$]
$0<\sigma_0^2:=\int_\R|v|^2 \,{\rm d}F_0(v) <+\infty$ \textit{and}
$\int_\R v\,{\rm d}F_0(v)=0$.
\item[$(H_\gamma)$]
\textit{If} $\gamma\in(0,1) \cup(1,2)$,
$F_0$ satisfies the conditions
$$
\lim_{x \to+\infty} x^\gamma\bigl(1-F_0(x)\bigr) =c_0^+<+\infty,\qquad
\lim_{x \to-\infty} |x|^\gamma F_0(x) =c_0^-<+\infty
$$
\textit{with} $c_0^++ c_0^->0$
\textit{and}, \textit{in addition}, $\int_\R v\,{\rm d}F_0(v)=0$ \textit{if  $\gamma\in (1,2)$.}
\end{itemize}
Further, we define the function $\hat g_\gamma:\R\mapsto \mathbb{C}$ by
\begin{equation}
\label{chaSta}
\hat g_\gamma(\xi):=
\begin{cases}
e^{ i m_0 \xi}, & \text{if } \gamma=1 \text{ and (a) of }(H_1) \text{ holds},\\
e^{ i m_0 \xi- \pi c_0^+ |\xi| }, & \text{if } \gamma=1 \text{ and (b) of }(H_1) \text{ holds},\\
e^{ - \sigma_0^2 |\xi|^2/2 }, &\text{if } \gamma=2 \text{ and }(H_2)\text{ holds},\\
e^{ -k_0 |\xi|^\gamma
(1-i \eta_0 \tan({\pi\gamma}/{2} )\operatorname{sign}\xi) }, &
\text{if }\gamma\in(0,1) \cup(1,2) \text{ and }(H_\gamma)\text{ holds},
\end{cases}
\end{equation}
where
%
\begin{equation*}
k_0 = (c_0^{+}+c_0^{-}) \frac{\pi}{2\Gamma(\gamma)\sin(\pi\gamma/2)},
\qquad \eta_0 = \frac{c_0^{+}-c_0^{-} }{c_0^{+}+c_0^{-}}.
\end{equation*}
Observe that the condition $(H_\gamma)$ is equivalent to the fact that the law of $X_0$ (the law of $X_0-m_0$ in case $H_1(b)$) is centered and (except case $H_1(a)$) belongs to the domain of normal attraction of a $\gamma$-stable law with the characteristic function $\hat g_\gamma$, see a concluding remark on p.~581 in \cite[Chapter XVIII.5]{Feller:1971}.

Now we formulate our hypotheses on the smoothing transform $\widehat Q$. Our first assumption is that the weights $(A_i)_{i=1,\ldots,N}$ are a.s.~positive. Next we define the function $\Phi:[0,\infty)\mapsto
\R\cup\{+\infty\}$ via
 $$
 \Phi(s) = \E\bigg[\sum_{i=1}^N A_i^s\bigg] - 1,\quad s\geq 0,
$$
and assume that $s_{\infty}>0$ where $s_{\infty} := \sup\{s\geq 0: \Phi(s) < \infty\}$. Note that the function $\Phi$ is smooth and convex on $(0,s_{\infty})$. The function
$$
\mu(s) = \frac{\Phi(s)}{s},\quad s > 0,
$$ is called spectral function, see \cite{bobylev}. Observe that $\mu(s)$ is equal to the tangent of the angle between the vector joining $(0,0)$ and $(s,\Phi(s))$ and the positive horizontal half-axis. Since $\Phi$ is strictly convex and smooth there exists exactly one point $\gamma^{\ast}$ minimizing the spectral function, and then the corresponding line is just tangent to the function $\Phi$ at point $(\gamma^{\ast}, \Phi(\gamma^{\ast}))$, see Fig.~\ref{fig_1}.  Moreover, $\mu(\gamma^{\ast}) = \Phi'(\gamma^{\ast})$.

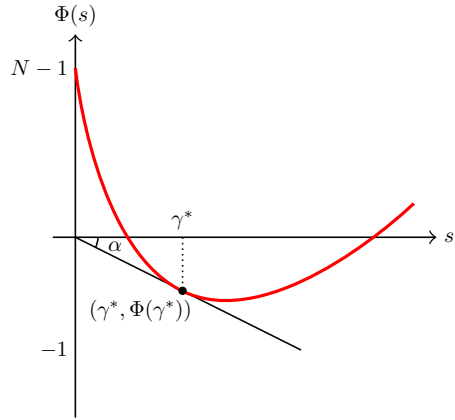
\begin{figure}[!hbtp]
\centering
 \scalebox{1.5}{\begin{tikzpicture}
\draw[->] (-0.2,0) -- (3.2,0) node[right,scale=0.5] {$s$};
\draw[->] (0,-1.6) -- (0,1.8) node[above,scale=0.5] {$\Phi(s)$};
\draw[-] (0,0) -- (2,-1) ;
\draw [densely dotted] (0.95,0) -- (0.95,-0.45) ;
\node[left,scale=0.5] (c) at (0,1.5) {$N-1$};
\node[left,scale=0.5] (c) at (0,-1) {$-1$};
\node[above,scale=0.5] (c) at (0.95,0) {$\gamma^{\ast}$};
\node[xshift=-12.0, yshift=-4.3,scale=0.5,radius=5.0] (c) at (1,-0.5) {$(\gamma^{\ast},\Phi(\gamma^{\ast}))$};
\draw (0.95,-0.475) circle (0.02);
\draw [thick,red] plot [smooth, tension=1] coordinates {(0,1.5) (1.0,-0.5) (3.0,0.3)};
\draw[fill=black] (0.95,-0.475) circle (0.03);
\draw (0.2,0) arc (0:-28:0.2) node[xshift=5.0,yshift=0.5,scale=0.5] {$\alpha$};;
\end{tikzpicture}}
\caption{Plot of the function $s\mapsto \Phi(s)$ (solid red) with $\tan\alpha=\mu(\gamma^{\ast})=\Phi^{\prime}(\gamma^{\ast})=\Phi(\gamma^{\ast})/\gamma^{\ast}$.}
\label{fig_1}
\end{figure}


In the series of papers \cite{BasettiLadelli2012,BasettiLadelliMatthes} Bassetti, Ladelli and Matthes found a probabilistic interpretation of the solution $\phi$ via labelled random trees. Assuming that $(H_{\gamma})$ holds for some $\gamma\in(0,2]$ and there exists $\delta>\gamma$ such that $\mu(\delta)<\mu(\gamma)<\infty$, it is shown in \cite[Theorem 2.2]{BasettiLadelli2012} that  $\phi(t, e^{-\mu(\gamma)t}\xi)$ converges to a nondegenerate limit being the characteristic function of the law of the limit of some positive martingale related to a family of random labelled trees. Clearly, if $\gamma=\argmin\mu(s)$  no such $\delta$ exists and, moreover, it can be checked that the corresponding martingale converges to $0$. As manifested in the title of the paper and motivated by Biggins and Kyprianou \cite{Biggins:Kyprianou:2005}, who considered the smoothing transform in the case  $\gamma^{\ast}=1$ and $\mu(\gamma^{\ast}) = 0$, we call this situation the {\it boundary case}.

\medskip

The main result of our paper is given by Theorem \ref{thm:main} below, and provides the correct normalization in the boundary case leading to a non-degenerate limit. As we will see, the right normalization involves a subexponential term and the limit is a fixed point of some smoothing transform.

\begin{thm}\label{thm:main}
Assume that for some $\gamma\in (0,2]$ the hypothesis $(H_{\gamma})$ is satisfied and
$$
\gamma=\argmin_{s\in(0,s_\infty)}\mu(s)=\gamma^{\ast}\in(0,s_\infty).
$$
Then there exists a probability measure $\rho_{\infty}$ such that the function $\phi$, the unique solution to \eqref{eqboltzivp}, satisfies
$$
  \lim_{t\to\infty}\phi\big(t,t^{\frac{1}{2\gamma}} e^{-\mu(\gamma)t}\xi\big) = w_{\infty}(\xi),\quad \xi\in\R,
  $$ where
   $w_{\infty}$ is the Fourier-Stieltjes transform of $\rho_{\infty}$. Moreover, $w_{\infty}$ has the following representation
  $w_{\infty}(\xi)=\E \widehat{g}_{\gamma}(\xi c_{\gamma} D^{1/\gamma}_{\infty})$,
  where $c_{\gamma}:=\left(\frac{2}{\pi\gamma^2\Phi^{\prime\prime}(\gamma)}\right)^{\frac{1}{2\gamma}}$ and $D_{\infty}$ is a.s.~positive random variable defined in Proposition \ref{prop: as} below and which satisfies the following stochastic fixed-point equation
\begin{equation}\label{eq:d_fixed_point}
  D_{\infty}\overset{d}{=}\mathcal{U}^{\Phi(\gamma)}\sum_{k=1}^{N}A_k^{\gamma}D_{\infty}^{(k)},
\end{equation}
where $(D_{\infty}^{(k)})_{k=1}^N$ are independent copies of $D_{\infty}$; $\mathcal{U}$ has a uniform distribution on $(0,1)$ and $(D_{\infty}^{(k)})_{k=1}^N$, $\mathcal{U}$ and $(A_1,\ldots,A_N)$ are independent.
\end{thm}


The rest of the paper is organized as follows. In Section \ref{sec:prob_interpretation} we describe a probabilistic representation
of the solution $\phi$, which  essentially reminds the construction in \cite{BasettiLadelli2012} but is more transparent and convenient for the analysis. Moreover, we reveal some further probabilistic structure behind this construction by pointing out a connection to Yule processes and branching random walks in continuous time. We strongly believe that the representation proposed in Section \ref{sec:prob_interpretation} is the most accurate probabilistic interpretation of the solution to equation \eqref{eqboltzivp}. In Section \ref{sec:biggins_convergence} we prove a convergence result for the Biggins martingale in continuous time branching random walk and explain the construction of the limiting measure $\rho_{\infty}$. The proof of Theorem \ref{thm:main} is given in Section \ref{sec:proof}.

\section{Probabilistic representation of the solution}\label{sec:prob_interpretation}

The solution to the equation \eqref{eqboltzivp} can be derived analytically  in terms of the Wild series \cite{Wild}, see also Kielek \cite{Kielek}. However, based on McKean's \cite{McKean} ideas, Bassetti, Ladelli and Matthes \cite{BasettiLadelli2012,BasettiLadelliMatthes} expressed the solution in a convenient probabilistic way. Ealier results on probabilistic representation can be found in \cite{Carlen+Carvalho+Gabetta:2000,Gabetta+Regazzini:2006,Gabetta+Regazzini:2008}.

The probabilistic construction of the solution $\phi$ using labelled $N$-ary random trees is given on pp.~1938--1939 of \cite{BasettiLadelli2012} see Proposition 3.2 therein, where the authors use among other a stochastic process called $(\nu_t)_{t\geq 0}$. However, it is defined as an arbitrary stochastic process with specified marginal distributions, see the top of p.~1939 in \cite{BasettiLadelli2012}. Even though such specification is sufficient for the asymptotic analysis of $\phi$, it leaves an open and interesting question of finding a correct interpretation and pathwise construction of $(\nu_t)_{t\geq 0}$. The main purpose of this subsection is to propose a natural representation of $(\nu_t)_{t\geq 0}$ and to provide an alternative form of Proposition 3.2 of \cite{BasettiLadelli2012} revealing the complete probabilistic structure of the solution $\phi$. As we will see, $\phi(t,\cdot)$ is nothing else but the characteristic function of a smoothing transform associated with a certain continuous-time branching random walk and applied to the distribution of $X_0$, see Proposition \ref{prop:prob_solution} below.

\subsection{Representation of the solutions and connection with branching random walks in continuous time.}

Let us recall that a Yule process $(\mathcal{Y}_t)_{t\geq 0}$ is a pure birth process which starts with one particle. After exponential time with parameter $1$ the original particle dies out and produces $N$ new particles. Every particle behaves as the original one, and the particles reproduce independently. The quantity $\mathcal{Y}_t$ is the number of particles at time $t\geq 0$. Denote by $F(s,t)$ the probability generating function of $\mathcal{Y}_t$, that is
$$
F(s,t)=\E s^{\mathcal{Y}_t},\quad t\geq 0,\quad |s|\leq 1.
$$
Using equations (5) and (6) on p.~106 in \cite{AthreyaNey}, see also example on p.~109 in the same reference, we obtain
$$
\frac{\partial F(s,t)}{\partial t} = F^{N}(s,t)-F(s,t),\quad t>0,\quad F(s,0)=s.
$$
By solving this differential equation, we get the explicit solution
\begin{equation}\label{eq:yule_process_gf}
F(s,t)=s\left(\frac{e^{-(N-1)t}}{1-s^{N-1}(1-e^{-(N-1)t})}\right)^{\frac{1}{N-1}},\quad t\geq 0,\quad |s|\leq 1.
\end{equation}
The full genealogical tree $\mathcal{T}_{\infty}$ of the Yule process $(\mathcal{Y}_t)_{t\geq 0}$ is an infinite $N$-ary random tree. For every fixed $T\geq 0$ the genealogical tree of $(\mathcal{Y}_t)_{t\in[0,T]}$ is a finite $N$-ary random tree with leaves representing the particles alive at time $T$ and internal nodes being the particles which have died out during $[0,T]$. Denote the number of latter particles by $\nu_T$. We have the following identity
\begin{equation}\label{eq:y_and_nu_connection}
\mathcal{Y}_t=(N-1)\nu_t+1,\quad t\geq 0.
\end{equation}
From this representation and formula \eqref{eq:yule_process_gf} we get
\begin{multline}\label{eq:nu_gen_func}
\E s^{\nu_t}=e^{-t}  \left(1-s\left(1-e^{-(N-1)t}\right)\right)^{- \frac{1}{N-1}}\\
=\sum_{k\geq 0}\frac{\Gamma(\frac{1}{N-1}+k)}{k!\Gamma(\frac{1}{N-1})}e^{-t}(1-e^{-(N-1)t})^k s^k,\quad t\geq 0,\quad |s|\leq 1,
\end{multline}
in full agreement with formula (3.4) in \cite{BasettiLadelli2012}. That is to say, the process $(\nu_t)_{t\geq 0}$ introduced in \cite{BasettiLadelli2012}, should be interpreted as the number of splits during the time interval $[0,t]$ in the Yule process $(\mathcal{Y}_t)_{t\geq 0}$. This interpretation of the distribution of $\nu_t$ is the starting point of our probabilistic construction of the solution $\phi$.

By adding to the definition of a Yule process the control over positions of particles, we obtain a continuous-time branching random walk introduced in \cite{Uchiyama:1982}. More precisely, let $\zeta=\sum_{k=1}^{N}\delta_{Z_k}$ be an arbitrary point process on $\R$, where $\delta_x$ denotes the Dirac point measure at $x\in\R$. In the continuous-time branching random walk the initial single particle is located at $0$. After an exponential time with parameter $1$ it dies out and gives birth to $N$ new particles which are placed at positions $(Z_1,\ldots, Z_N)$. These particles reproduce independently in the same fashion as their mother. In particular, if at any time a particle $v$ located at some $x\in\R$ splits, its children are placed at $x+Z_1(v),\ldots,x+Z_N(v)$, where $\zeta^{(v)}:=\sum_{k=1}^{N}\delta_{Z_k(v)}$ is an independent copy of $\zeta$. In what follows we only consider branching random walks with deterministic number of children of every particle. Clearly, the number of particles in such a continuous-time branching random walk at time $t\geq 0$ is just $\mathcal{Y}_t$. Denote the locations of particles present at time $t$ by $z_{1,t},z_{2,t},\ldots,z_{\mathcal{Y}_t,t}$. The continuous-time branching random walk is formally defined as the measure-valued stochastic process
$$
\mathcal{Z}_t:=\sum_{k=1}^{\mathcal{Y}_t}\delta_{z_{k,t}},\quad t\geq 0.
$$
It will be important that the process $(\mathcal{Z}_t)_{t\geq 0}$ satisfies the following branching relation:
\begin{equation}\label{eq:brw_branching_property}
\mathcal{Z}_{t+s}(\cdot)\overset{d}{=}\sum_{k=1}^{\mathcal{Y}_t}\mathcal{Z}_{s}^{(k)}(\cdot - z_{k,t}),\quad t,s\geq 0,
\end{equation}
where $(\mathcal{Z}_{t}^{(k)})_{t\geq 0}$ for $k\in\N$ are independent copies of $(\mathcal{Z}_t)_{t\geq 0}$.

Finally, given a continuous-time branching random walk $(\mathcal{Z}_t)_{t\geq 0}$, the associated family of smoothing transforms $(\mathcal{L}^{(\gamma)}_t)_{t\geq 0}$ on the space of probability distributions on $\R$ is defined by
$$
\mathcal{L}^{(\gamma)}_t({\rm distr}(U))={\rm distr}\left(\sum_{k=1}^{\mathcal{Y}_t}e^{\gamma z_{k,t}}U_k\right),
$$
where $(U_k)_{k\geq 1}$ are independent copies of a random variable $U$ and $\gamma\in\mathbb{C}$ is a parameter. By slightly abusing notation we write $\mathcal{L}^{(\gamma)}_t(U)$ instead  of $\mathcal{L}^{(\gamma)}_t({\rm distr}(U))$. We also suppress the index $\gamma$ if it is equal to $1$ by writing $\mathcal{L}_t$ instead of $\mathcal{L}^{(1)}_t$.

We are ready to state the main result of this subsection, namely the probabilistic representation of the solution $\phi$ to kinetic-type equation \eqref{eqboltzivp}. Assume that on the probability space $(\Omega,\mathcal{F},\P)$ the following two objects are defined:
\begin{itemize}
\item the continuous time branching random walk $(\mathcal{Z}_{t})_{t\geq 0}$ with the displacement process $\zeta:=\sum_{k=1}^{N}\delta_{\log A_k}$:
$$
\mathcal{Z}_t:=\sum_{k=1}^{\mathcal{Y}_t}\delta_{z_{k,t}},\quad t\geq 0.
$$
\item the sequence $(X_k)_{k\geq 1}$ of independent random variables with common distribution function $F_0$, which is also independent of $(\mathcal{Z}_{t})_{t\geq 0}$.
\end{itemize}

\begin{prop}\label{prop:prob_solution}
Equation \eqref{eqboltzivp} has a unique solution $\phi(t,\cdot)$ which is given by
\begin{equation}\label{eq:prob_solution_main}
\phi(t,\xi)=\E\exp\left(i\xi \left(\sum_{k=1}^{\mathcal{Y}_t}e^{z_{k,t}}X_k\right)\right),\quad t\geq 0,\quad \xi\in\R,
\end{equation}
that is $\phi(t,\cdot)$ is the characteristic function of the random variable $\mathcal{L}_t(X_0)$, where $\mathcal{L}_t$ is the smoothing transform associated with the continuous-time branching random walk $(\mathcal{Z}_{t})_{t\geq 0}$.
\end{prop}
\begin{proof}
Uniqueness of the solution follows from a standard use of the Picard-Lindel\"of theorem, see e.g. proof of Proposition 2.2 in \cite{BasettiLadelliMatthes2}. Thus it is enough to show that the right-hand side of \eqref{eq:prob_solution_main} satisfies \eqref{eqboltzivp}.
To this end, denote the the right-hand side of \eqref{eq:prob_solution_main} by $\psi(t,\xi)$ and write
$$
\psi(t,\xi)=\E\left[\prod_{k=1}^{\mathcal{Y}_t}
\phi_0(\xi e^{z_{k,t}})\right]=\E\exp\left(\int_{\R}\log \phi_0(\xi e^y)\mathcal{Z}_t({\rm d}y)\right),\quad t\geq 0,\quad \xi\in\R.
$$

Firstly, let us show that $t\mapsto \psi(t,\xi)$ is continuous for every fixed $\xi$. For $t,s\geq 0$ we can write
\begin{multline*}
|\psi(t,\xi)-\psi(s,\xi)|\leq  2 \P\{\text{there are splits during } [t\wedge s,t\vee s]\}=2 \E (1-e^{-\mathcal{Y}_{t\wedge s}|t-s|})\to 0,
\end{multline*}
as $s\to t$, by the dominated convergence theorem and the observation $\mathcal{Y}_{t}<\infty$ a.s.

Further, for $t\geq 0$, let $\mathcal{F}_t\subset \mathcal{F}$ be the $\sigma$-algebra generated by $(\mathcal{Z}_s)_{s\in[0,t]}$. For $h\geq 0$, using formula \eqref{eq:brw_branching_property}, we obtain
\begin{align*}
\psi(t+h,\xi)&=\E\left(\E\left(\exp\left(\int_{\R}\log \phi_0(\xi e^y)\mathcal{Z}_{t+h}({\rm d}y)\right)\Big|\mathcal{F}_h\right)\right)\\
&=\E\left(\E\left(\prod_{k=1}^{\mathcal{Y}_h}\exp\left(\int_{\R}\log \phi_0(\xi e^y e^{z_{k,h}})\mathcal{Z}^{(k)}_{t}({\rm d}y)\right)\Big|\mathcal{F}_h\right)\right)\\
&=\E\left(\prod_{k=1}^{\mathcal{Y}_h}\E\left(\exp\left(\int_{\R}\log \phi_0(\xi e^y e^{z_{k,h}})\mathcal{Z}^{(k)}_{t}({\rm d}y)\right)\Big|\mathcal{F}_h\right)\right)=\E\left(\prod_{k=1}^{\mathcal{Y}_h}\psi(t,\xi e^{z_{k,h}})\right).
\end{align*}
The probability of having two or more splits in the branching random walk $(\mathcal{Z}_{t})_{t\geq 0}$ during $[0,h]$ is $o(h)$ as $h\to +0$, whence
\begin{align*}
\psi(t+h,\xi)&=\E\left(\prod_{k=1}^{\mathcal{Y}_h}\psi(t,\xi e^{z_{k,h}})\right)\\
&=\psi(t,\xi)\P\{\text{there are no splits during } [0,h]\}\\
&+\E\left(\prod_{j=1}^{N}\psi(t,\xi A_j)\right)\P\{\text{there is exactly one split during } [0,h]\}+o(h)\\
&=\psi(t,\xi)e^{-h}+\E\left(\prod_{j=1}^{N}\psi(t,\xi A_j)\right)h+o(h).
\end{align*}
Likewise, we can write for $h\geq  0$ and $t\geq h$:
\begin{align*}
\psi(t,\xi)&=\E\left(\prod_{k=1}^{\mathcal{Y}_h}\psi(t-h,\xi e^{z_{k,h}})\right)\\
&=\psi(t-h,\xi)\P\{\text{there are no splits during } [0,h]\}\\
&+\E\left(\prod_{j=1}^{N}\psi(t-h,\xi A_j)\right)\P\{\text{there is exactly one split during } [0,h]\}+o(h)\\
&=\psi(t-h,\xi)e^{-h}+\E\left(\prod_{j=1}^{N}\psi(t-h,\xi A_j)\right)h+o(h).
\end{align*}
Rearranging the terms and sending $h\to +0$ shows that
$$
\frac{\partial \psi(t,\xi)}{\partial t}+\psi(t,\xi)=\E\left(\prod_{j=1}^{N}\psi(t,\xi A_j)\right)=\widehat{Q}(\psi(t,\cdot),\ldots,\psi(t,\cdot)),\quad t>0,
$$
by the dominated convergence theorem and continuity of $t\mapsto \psi(t,\xi)$ (this is required for the left derivative). Therefore, $\psi$ is a solution to \eqref{eqboltzivp}. Since the solution is unique and $\psi(0,\xi)=\phi_0(\xi)=\phi(0,\xi)$, we infer $\psi(t,\xi)\equiv\phi(t,\xi)$. The proof is complete.
\end{proof}




\begin{rem}
Let us now compare our Proposition \ref{prop:prob_solution} with Proposition 3.2 in \cite{BasettiLadelli2012} in more details. Proposition 3.2 in \cite{BasettiLadelli2012}
states that the unique solution $\phi$ to \eqref{eqboltzivp} is
$$
\phi(t,\xi)=\int_\R e^{i\xi v}\rho_t({\rm d}v)=\int_\R e^{i\xi v}\P\{W_{\nu_t}\in {\rm d}v\},
$$
where $\nu_t$ is an integer-valued random variable with the generating function \eqref{eq:nu_gen_func} and which is independent of $(W_n)_{n\geq 0}$. For $n=0,1,2,\ldots$, $W_n$ is defined by a sum
$$
W_n:=\sum_{i=1}^{(N-1)n+1} \omega(v_{i,n}) X_{v_{i,n}},
$$
where $v_{i,n}$, $i=1,\ldots,(N-1)n+1$ are the leaves of a random labelled $N$-ary recursive tree $\mathcal{T}_n$ after $n$ steps, $\omega(v)$ is the weight of a leaf $v\in\mathcal{T}_n$ (the product of all labels on the unique path from the root to $v$), and $(X_{v})$ is a family of independent random variables with common distribution function $F_0$ which is also independent of the random labelled tree $\mathcal{T}_n$. 
Our construction described above unifies and reinterprets all the aforementioned ingredients: the sequence of random $N$-ary trees $(\mathcal{T}_n)$, the labels of their nodes and the subordination time $\nu_t$ via a single object, the continuous-time branching random walk $(\mathcal{Z}_t)_{t\geq 0}$. We summarize the above observations in Table \ref{tab1}.
\begin{table}[!htbp]
\caption{Comparison of two probabilistic constructions of $\phi$}
\begin{tabular}{p{7cm}p{7cm}}
\hline
The construction in \cite{BasettiLadelli2012,BasettiLadelliMatthes} & A counterpart in our construction\\
\hline
the sequence of random $N$-ary recursive trees $(\mathcal{T}_n)_{n\geq 0}$ & the skeleton of the Yule process $(\mathcal{Y}_{t})_{t\geq 0}$ pertained to $(\mathcal{Z}_{t})_{t\geq 0}$ and observed at splitting times\\
\hline
labels of the nodes in the trees $(\mathcal{T}_n)_{n\geq 0}$ & relative displacements of the particles in $(\mathcal{Z}_{t})_{t\geq 0}$\\
\hline
random variables $\nu_t$, $t\geq 0$ & random process $(\nu_t)_{t\geq 0}$, the number of splits in
$(\mathcal{Z}_{s})_{s\geq 0}$ (or $(\mathcal{Y}_{t})_{t\geq 0}$) on $[0,t]$\\
\hline
\end{tabular}
\label{tab1}
\end{table}

The connection between random $N$-ary trees, Yule processes and branching random walks, is by no means new and have already been observed in probabilistic literature, see for example \cite{Chauvin+Klein+Marckert+Rouault:2004} for the case of binary search trees. Recently this connection has been extensively exploited in the analysis of profiles of random trees in \cite{Kabluchko+Marynych+Sulzbach:2017}. 

\end{rem}

\begin{rem}
As has been pointed out by the referee our Proposition \ref{prop:prob_solution} remains valid also with random $N$ such that $\E N\in (1,\infty)$. The latter condition guarantees that $(\mathcal{Z}_t)_{t\geq 0}$ does not explode and has a positive survival probability. On the other hand, probabilistic construction used in \cite{BasettiLadelli2012,BasettiLadelliMatthes} does not seem to have a direct analogue for random $N$ due to a lack of explicit distribution for $\nu_t$ for a fixed $t>0$.
\end{rem}

Last but not least, we would like to emphasize that the main advantage of Proposition \ref{prop:prob_solution} is its generality. It allows one to translate limit theorems for the smoothing transform $\mathcal{L}^{(\gamma)}_t(X_0)$, as $t\to\infty$, to the corresponding asymptotics for the solution of \eqref{eqboltzivp}, as $t\to\infty$. In particular, Proposition \ref{prop:prob_solution} is useful not only in the case considered in our paper~--~the boundary case~--~but also in other situations. As we will see in the next sections, limit theorems for $\mathcal{L}^{(\gamma)}_t(X_0)$ are intimately connected with convergence in probability of a so-called Biggins martingale for the continuous-time branching random walk $(\mathcal{Z}_t)_{t\geq 0}$.

\section{Convergence of the continuous-time Biggins martingale in the boundary case.}\label{sec:biggins_convergence}
For every $\gamma\in [0,s_{\infty})$, put
$$
\mathcal{M}_t(\gamma):= e^{-\Phi(\gamma)t}\sum_{k=1}^{\mathcal{Y}_t}e^{\gamma z_{k,t}},\quad t\geq 0,
$$
and note that by formula (5.1) in \cite{Biggins:1992} we have
\begin{equation}\label{eq:biggins_martingale_exp}
 \E \mathcal{M}_t(\gamma)=1.
\end{equation}
The stochastic process $(\mathcal{M}_t(\gamma))_{t\geq 0}$ is a martingale and is called {\it continuous-time Biggins martingale}.

If $\gamma=\gamma^{\ast}=\argmin_{s\in [0,s_{\infty})}\mu(s)$ and $\gamma^{\ast}<s_{\infty}$, then the Biggins martingale $(\mathcal{M}_t(\gamma^{\ast}))_{t\geq 0}$ converges to zero a.s. For the discrete-time Biggins martingale this fact is well-known, see, for example, Lemma 5 in \cite{Biggins:1977}, and for the continuous-time Biggins martingale it follows from Theorem 1.1 of the recent paper \cite{Bertoin+Mallein:2018} as well as from Proposition \ref{prop: as}(i) below.

\begin{prop}\label{prop: as} Assume that $\gamma^{\ast}\in (0,s_\infty)$. The following limit relations hold true.
\begin{itemize}
\item[(i)] As $t\to\infty$ we have
  \begin{equation}\label{eq:pp1}
 \sqrt{t} \mathcal{M}_t(\gamma^{\ast})
=\sqrt{t} \sum_{k=1}^{\mathcal{Y}_t}e^{\gamma^{\ast}z^{\circ}_{k,t}}
\overset{\P}{\to}\sqrt{\frac{2}{\pi (\gamma^{\ast})^2\Phi^{\prime\prime}(\gamma^{\ast})}}D_{\infty},
  \end{equation}
where $z^{\circ}_{k,t} = z_{k,t} - t \mu(\gamma^*)$,  $D_{\infty}$ is the a.s. limit of the derivative martingale
\begin{equation}\label{eq:prop1_derivative_martingale}
\mathcal{D}_t(\gamma^{\ast}):=\sum_{k=1}^{\mathcal{Y}_t}e^{\gamma^{\ast}z^{\circ}_{k,t}}z^{\circ}_{k,t},\quad t\geq 0.
\end{equation}
The random variable $D_{\infty}$ is a.s.~positive and satisfies \eqref{eq:d_fixed_point}.
\item[(ii)] Moreover,
  \begin{equation}\label{eq:pp2}
\sqrt{t}  \max_{k=1,\ldots,\mathcal{Y}_t}e^{\gamma^{\ast}z^{\circ}_{k,t}} \overset{\P}{\to} 0,\quad t\to\infty.
  \end{equation}
  \end{itemize}
\end{prop}

The derivation of Proposition \ref{prop: as} utilizes ideas borrowed from \cite{Dadoun:2017}, where part (i) has been stated without a proof in Remark 2.11(iii). Firstly, we obtain two auxiliary lemmas which show that the Biggins martingale $(\mathcal{M}_t(\gamma^{\ast}))_{t\geq 0}$ is in the {\it boundary case}. In particular, this implies that every $\theta$-skeleton, that is the discrete-time Biggins martingale $(\mathcal{M}_{n\theta}(\gamma^{\ast}))_{n\ge 0}$, $\theta>0$, is also in the boundary case. Thereafter, we apply the corresponding theorem by A\"id\'ekon and Shi \cite{AidekonShi2014},  who found the appropriate  normalization for the discrete-time Biggins martingales in the boundary case, to our $\theta$-skeletons and then pass to the continuous parameter with the aid of the Croft--Kingman lemma.

\begin{lem}\label{lem:1} Assume that $\gamma^{\ast}\in (0,s_\infty)$. For every $t\geq 0$ we have
  \begin{equation}\label{eq:lem1_claims}
    \E\bigg[ \sum_{k=1}^{\mathcal{Y}_t}   e^{\gamma^{\ast} z^{\circ}_{k,t}} z^{\circ}_{k,t} \bigg] = 0\quad\text{and}\quad \E\bigg[ \sum_{k=1}^{\mathcal{Y}_t}e^{\gamma^{\ast} z^{\circ}_{k,t}}(z^{\circ}_{k,t})^2 \bigg] = t\Phi^{\prime\prime}(\gamma^{\ast}).
  \end{equation}
\end{lem}
\begin{proof}
Fix $\varepsilon\in (0,\gamma^{\ast})$ such that $\gamma^{\ast}+\varepsilon<s_{\infty}$. Let us show that for every fixed $t\geq 0$ the following holds:
$$
\E\bigg[ \sum_{k=1}^{\mathcal{Y}_t}e^{\gamma^{\ast} z_{k,t}}z_{k,t} \bigg]=\frac{\partial }{\partial \gamma}\left[\E\int_{\R} e^{\gamma y}\mathcal{Z}_t({\rm d}y)\right]\Bigg|_{\gamma=\gamma^{\ast}}.
$$
To this end, it is enough to check that the partial derivative on the right-hand side can be moved inside the expectation and the integration signs. But this is a simple consequence of the dominated convergence theorem, since
$$
\lim_{\Delta \to 0}\int_{\Omega}\int_{\R} \frac{e^{(\gamma^{\ast}+\Delta) y}-e^{\gamma^{\ast} y}}{\Delta}\mathcal{Z}_t({\rm d}y){\rm d}\P=\lim_{\Delta \to 0}\int_{\Omega}\int_{\R} \frac{e^{\Delta y}-1}{\Delta}e^{\gamma^{\ast} y}\mathcal{Z}_t({\rm d}y){\rm d}\P,
$$
and the absolute value of the integrand is bounded by the integrable function
$$
y\mapsto e^{(\gamma^{\ast}+\varepsilon) y}{\bf 1}_{\{y\geq 0\}}+e^{(\gamma^{\ast}-\varepsilon) y}{\bf 1}_{\{y<0\}}
$$
for sufficiently small $\Delta$ and all $y\in\R$.

Using formula \eqref{eq:biggins_martingale_exp} we derive
$$
\E\bigg[ \sum_{k=1}^{\mathcal{Y}_t}e^{\gamma^{\ast} z_{k,t}} z_{k,t} \bigg]=\frac{\partial}{\partial \gamma}\left(e^{t\Phi(\gamma)}\right)\Bigg|_{\gamma=\gamma^{\ast}}=t\Phi^{\prime}(\gamma^{\ast})e^{t\Phi(\gamma^{\ast})},\quad t\geq 0.
$$
This immediately yields
$$
    \E\bigg[ \sum_{k=1}^{\mathcal{Y}_t}e^{\gamma^{\ast} z^{\circ}_{k,t}}z^{\circ}_{k,t} \bigg]=e^{-t\Phi(\gamma^{\ast})}\left(t\Phi^{\prime}(\gamma^{\ast})e^{t\Phi(\gamma^{\ast})}\right)
    -t  \mu(\gamma^{\ast})\E \mathcal{M}_t(\gamma^{\ast})=t\left(\Phi^{\prime}(\gamma^{\ast})- \mu(\gamma^{\ast})\right) = 0.
$$
The second claim in \eqref{eq:lem1_claims} follows from the formula
$$
\E\bigg[ \sum_{k=1}^{\mathcal{Y}_t}e^{\gamma^{\ast} z_{k,t}} z^2_{k,t} \bigg]=\frac{\partial^2 }{\partial \gamma^2}\left[\E\int_{\R} e^{\gamma y}\mathcal{Z}_t({\rm d}y)\right]\Bigg|_{\gamma=\gamma^{\ast}},
$$
which can be proved similarly. The proof is complete.
\end{proof}

\begin{rem}
As has been pointed out by the referee Lemma 3.6 also follows from the many-to-one lemma for continuous-time branching random walks.
\end{rem}

\begin{lem}\label{lem:2}
Assume that $\gamma^{\ast}<s_{\infty}$. Then for every fixed $t\geq 0$ and $\delta>0$ such that $(1+\delta)\gamma^{\ast} < s_{\infty}$ we have
  \begin{equation*}
\E\bigg[ \bigg(  \sum_{k=1}^{\mathcal{Y}_t}  e^{\gamma^{\ast} z^{\circ}_{k,t}}\bigg)^{1+\delta} \bigg] <\infty\quad\text{and}\quad \E\bigg[ \bigg(  \sum_{k=1}^{\mathcal{Y}_t}  e^{\gamma^{\ast} z^{\circ}_{k,t}}( z_{k,t})_{+} \bigg)^{1+\delta} \bigg]<\infty,
  \end{equation*}
 where $x_{+}:=\max(x,0)$.
\end{lem}

\begin{proof}


Let us prove the first claim. Using the inequality
$$
\left(\sum_{k=1}^{n}x_k\right)^{1+\delta}\leq n^{\delta}\left(\sum_{k=1}^{n}x_k^{1+\delta}\right)
$$
which holds for $n\in\N$ and arbitrary nonnegative reals $x_1,x_2,\ldots,x_n$, we infer
$$
\E\bigg( \sum_{k=1}^{\mathcal{Y}_t}  e^{\gamma^{\ast} z_{k,t}} \bigg)^{1+\delta}\leq
\E\left( \mathcal{Y}_t^{\delta} \sum_{k=1}^{ \mathcal{Y}_t}e^{(1+\delta)\gamma^{\ast} z_{k,t}}\right)=\sum_{k=1}^{\infty}\E\left( \mathcal{Y}_t^{\delta} e^{(1+\delta)\gamma^{\ast} z_{k,t}}{\bf 1}_{\{k\leq \mathcal{Y}_t\}}\right).
$$
Pick $p>1$ such that $p(1+\delta)\gamma^{\ast}<s_{\infty}$ and $q>1$ such that $1/p+1/q=1$. By H\"{o}lder's inequality we obtain
$$
\E\left( \mathcal{Y}_t^{\delta} e^{(1+\delta)\gamma^{\ast} z_{k,t}}{\bf 1}_{\{k\leq \mathcal{Y}_t\}}\right)\leq \left(\E e^{p(1+\delta)\gamma^{\ast} z_{k,t}}{\bf 1}_{\{k\leq \mathcal{Y}_t\}}\right)^{1/p}\left(\E \mathcal{Y}_t^{q \delta}{\bf 1}_{\{k\leq \mathcal{Y}_t\}}\right)^{1/q},
$$
and thereupon
\begin{multline}\label{eq:lem2_proof1}
\E\bigg( \sum_{k=1}^{\mathcal{Y}_t}  e^{\gamma^{\ast}z_{k,t}} \bigg)^{1+\delta}\leq \sum_{k=1}^{\infty}\left(\E e^{p(1+\delta)\gamma^{\ast}z_{k,t}}{\bf 1}_{\{k\leq \mathcal{Y}_t\}}\right)^{1/p}\left(\E\mathcal{Y}_t^{q \delta}{\bf 1}_{\{k\leq \mathcal{Y}_t\}}\right)^{1/q}\\
\leq \left(\sum_{k=1}^{\infty}\E e^{p(1+\delta)\gamma^{\ast} z_{k,t}}{\bf 1}_{\{k\leq \mathcal{Y}_t\}}\right)^{1/p}\left(\sum_{k=1}^{\infty}\E \mathcal{Y}_t^{q \delta}{\bf 1}_{\{k\leq \mathcal{Y}_t\}}\right)^{1/q},\end{multline}
where the last passage is a consequence of H\"{o}lder's inequality for series. The first factor on the right-hand side is finite because $p(1+\delta)\gamma^{\ast}<s_{\infty}$ and
\begin{multline*}
\sum_{k=1}^{\infty}\E e^{p(1+\delta)\gamma^{\ast} z_{k,t}}{\bf 1}_{\{k\leq \mathcal{Y}_t\}}=\E \left(\sum_{k=1}^{\mathcal{Y}_t}e^{p(1+\delta)\gamma^{\ast} z_{k,t}}\right)=\E \int_{\R}e^{p(1+\delta)\gamma^{\ast}y}\mathcal{Z}_t({\rm d}y)\\
=e^{\Phi(p(1+\delta)\gamma^{\ast})t}\E\mathcal{M}_t(p(1+\delta)\gamma^{\ast})=e^{\Phi(p(1+\delta)\gamma^{\ast})t}<\infty.
\end{multline*}
Formulae \eqref{eq:y_and_nu_connection} and \eqref{eq:nu_gen_func} imply that $\mathcal{Y}_t$ has exponential moment of some positive order for every fixed $t$. Therefore,
$$
\sum_{k=1}^{\infty}\E\mathcal{Y}_t^{q \delta}{\bf 1}_{\{k\leq \mathcal{Y}_t\}}=\E\mathcal{Y}_t^{q \delta+1}<\infty
$$
and the proof  of the first claim is complete.

To prove the second inequality we use exactly the same arguments to get the upper bound
\begin{multline*}
\E\bigg( \sum_{k=1}^{\mathcal{Y}_t}  e^{\gamma^{\ast}z_{k,t}}(z_{k,t})_{+} \bigg)^{1+\delta}\\
\leq \left(\sum_{k=1}^{\infty}\E e^{p(1+\delta)\gamma^{\ast}z_{k,t}}(z_{k,t})^{p(1+\delta)}_{+}{\bf 1}_{\{k \leq \mathcal{Y}_t\}}\right)^{1/p}\left(\sum_{k=1}^{\infty}\E\mathcal{Y}_t^{q \delta}{\bf 1}_{\{k\leq \mathcal{Y}_t\}}\right)^{1/q}.
\end{multline*}
It remains to note that
$$
\sum_{k=1}^{\infty}\E e^{p(1+\delta)\gamma^{\ast}z_{k,t}}(z_{k,t})^{p(1+\delta)}_{+}{\bf 1}_{\{k \leq \mathcal{Y}_t\}}=\E \int_{\R}e^{p(1+\delta)\gamma^{\ast}y}y^{p(1+\delta)}_+\mathcal{Z}_t({\rm d}y)<\infty,
$$
since $p(1+\delta)\gamma^{\ast}<s_{\infty}$. The proof is complete.
\end{proof}

\begin{proof}[Proof of Proposition \ref{prop: as}]

{\sc Proof of part (i).} Fix $\theta>0$. Define a point process
$$
 \Xi:=\sum_{k=1}^{\mathcal{Y}_{\theta}}\delta_{-\gamma^{\ast}z^{\circ}_{k,\theta}},
$$
and consider a discrete-time branching random walk $(\mathcal{Z}_n(\theta))_{n=0,1,2,\ldots}$, where
$$
\mathcal{Z}_n(\theta):=\sum_{k=1}^{\mathcal{Y}_{n\theta}}\delta_{ -\gamma^{\ast} z^{\circ}_{k,n\theta}},\quad n=0,1,2,\ldots.
$$
The discrete-time branching random walk $(\mathcal{Z}_k(\theta))$ has the displacement process $\Xi$ and satisfies the following three conditions:
\begin{multline}\label{eq:boundary_case_skeleton}
\E\left(\int_{\R}e^{-y}\mathcal{Z}_1(\theta)({\rm d}y)\right)=1,\quad \E\left(\int_{\R}e^{-y}y\mathcal{Z}_1(\theta)({\rm d}y)\right)=0\quad \text{and}\\
\E\left(\int_{\R}e^{-y}y^2\mathcal{Z}_1(\theta)({\rm d}y)\right)=\theta(\gamma^{\ast})^2\Phi^{\prime\prime}(\gamma^{\ast})<\infty,
\end{multline}
where the last two relations are secured by Lemma \ref{lem:1}. Moreover, Lemma \ref{lem:2} yields
$$
\E\left(\int_{\R}e^{-y}\mathcal{Z}_1(\theta)({\rm d}y)\right)^{1+\delta}<\infty\quad\text{and}\quad \E\left(\int_{\R}e^{-y}y_{+}\mathcal{Z}_1(\theta)({\rm d}y)\right)^{1+\delta}<\infty,
$$
whence conditions (5.3) in  \cite{Shi:2015} hold. Therefore, Assumption (H) in the same reference holds for the discrete-time branching random walk $(\mathcal{Z}_n(\theta))_{n=0,1,2,\ldots}$ for every fixed $\theta>0$. In particular, by Theorem 5.29 in \cite{Shi:2015}, see also Theorem 1.1 in \cite{AidekonShi2014}, we have
\begin{multline}\label{eq:conv_along_integers}
\sqrt{n} \mathcal{M}_{n}(\gamma^{\ast})=\sqrt{n} \sum_{k=1}^{\mathcal{Y}_{n}}e^{\gamma^{\ast}z^{\circ}_{k,n}}
=\sqrt{n} \int_{\R}e^{-y}\mathcal{Z}_{n}(1)({\rm d}y)\\
\overset{\P}{\to} \sqrt{\frac{2}{\pi (\gamma^{\ast})^2 \Phi^{\prime\prime}(\gamma^{\ast})}}D_{\infty}=:D,\quad n\to\infty,
\end{multline}
where $D$ is a.s. positive, because in our settings the process does not extinct with probability one. It remains to show the convergence in probability to $D$ along $t\to\infty$, $t\in\R$. This can be accomplished by adopting the argument given on p.~47 in \cite{Biggins+Kyprianou:1996} as follows. From \eqref{eq:conv_along_integers} we know that for every fixed $x>0$
$$
\left(\sqrt{n+1} \mathcal{M}_{n+1}(\gamma^{\ast})-\sqrt{n} \mathcal{M}_{n}(\gamma^{\ast})\right)\1_{\{\sqrt{n} \mathcal{M}_{n}(\gamma^{\ast})\leq x\}}\overset{\P}{\to} 0,\quad n\to\infty,
$$
and therefore by the dominated convergence theorem	we have for every $u\geq 0$
\begin{equation*}
\E \exp\left(-u\left(\left(\sqrt{n+1} \mathcal{M}_{n+1}(\gamma^{\ast})-\sqrt{n} \mathcal{M}_{n}(\gamma^{\ast})\right)\1_{\{\sqrt{n} \mathcal{M}_{n}(\gamma^{\ast})\leq x\}}\right)\right)\to 1,\quad n\to\infty.
\end{equation*}
Further, by the martingale property of $(\mathcal{M}_t(\gamma^{\ast}))$ and applying Jensen's inequality twice to the convex function $x\mapsto \exp(-ux)$ we obtain for every $t\geq 0$
\begin{align*}
& \E \exp\left(-u\left(\left(\sqrt{[t]+1} \mathcal{M}_{[t]+1}(\gamma^{\ast})-\sqrt{[t]} \mathcal{M}_{[t]}(\gamma^{\ast})\right)\1_{\{\sqrt{[t]} \mathcal{M}_{[t]}(\gamma^{\ast})\leq x\}}\right)\right)\\
&=\E \left[\E \left\{\exp\left(-u\left(\left(\sqrt{[t]+1} \mathcal{M}_{[t]+1}(\gamma^{\ast})-\sqrt{[t]} \mathcal{M}_{[t]}(\gamma^{\ast})\right)\1_{\{\sqrt{[t]} \mathcal{M}_{[t]}(\gamma^{\ast})\leq x\}}\right)\right)\Big|\mathcal{F}_t \right\}\right]\\
&\geq \E \left\{\exp\left(-u\left(\left(\sqrt{[t]+1} \mathcal{M}_{t}(\gamma^{\ast})-\sqrt{[t]} \mathcal{M}_{[t]}(\gamma^{\ast})\right)\1_{\{\sqrt{[t]} \mathcal{M}_{[t]}(\gamma^{\ast})\leq x\}}\right)\right)\right\}\\
&= \E \left[\E \left\{\exp\left(-u\left(\left(\sqrt{[t]+1} \mathcal{M}_{t}(\gamma^{\ast})-\sqrt{[t]} \mathcal{M}_{[t]}(\gamma^{\ast})\right)\1_{\{\sqrt{[t]} \mathcal{M}_{[t]}(\gamma^{\ast})\leq x\}}\right)\right)\Big|\mathcal{F}_{[t]}\right\}\right]\\
&\geq \E \left\{\exp\left(-u\left(\left(\sqrt{[t]+1} \mathcal{M}_{[t]}(\gamma^{\ast})-\sqrt{[t]} \mathcal{M}_{[t]}(\gamma^{\ast})\right)\1_{\{\sqrt{[t]} \mathcal{M}_{[t]}(\gamma^{\ast})\leq x\}}\right)\right)\right\}.
\end{align*}
Sending $t\to\infty$ in the above inequalities we obtain
\begin{equation}\label{eq:conv_along_integers2}
\left(\sqrt{[t]+1} \mathcal{M}_{t}(\gamma^{\ast})-\sqrt{[t]} \mathcal{M}_{[t]}(\gamma^{\ast})\right)\1_{\{\sqrt{[t]} \mathcal{M}_{[t]}(\gamma^{\ast})\leq x\}}\overset{\P}{\to} 0,\quad t\to\infty.
\end{equation}
By the triangle inequality
\begin{align*}
\left|\sqrt{t}\mathcal{M}_{t}(\gamma^{\ast})-D\right|&\leq \left|\sqrt{t} \mathcal{M}_{t}(\gamma^{\ast})-\sqrt{[t]+1} \mathcal{M}_{t}(\gamma^{\ast})\right|\\
& + \left|\sqrt{[t]+1} \mathcal{M}_{t}(\gamma^{\ast})-\sqrt{[t]} \mathcal{M}_{[t]}(\gamma^{\ast})\right|\1_{\{\sqrt{[t]} \mathcal{M}_{[t]}(\gamma^{\ast})\leq x\}}\\
&+ \left|\sqrt{[t]+1} \mathcal{M}_{t}(\gamma^{\ast})-\sqrt{[t]} \mathcal{M}_{[t]}(\gamma^{\ast})\right|\1_{\{\sqrt{[t]} \mathcal{M}_{[t]}(\gamma^{\ast})> x\}}\\
&+ |\sqrt{[t]} \mathcal{M}_{[t]}(\gamma^{\ast})-D|.
\end{align*}
The second and fourth summands converge to zero in probability as $t\to\infty$ by \eqref{eq:conv_along_integers2} and \eqref{eq:conv_along_integers}, respectively. The first summand  does this by Markov's inequality since $\sqrt{t+1}-\sqrt{t}\to 0$ as $t\to\infty$. The probability that the third summand is larger than some $\delta>0$ is bounded from above by $\P\{\sqrt{[t]} \mathcal{M}_{[t]}(\gamma^{\ast})> x\}$ which can be made arbitrarily small by choosing $x$ large enough in view of \eqref{eq:conv_along_integers} and a.s. finiteness of $D$. This completes the proof of convergence in part (i).

Let us show that $D_{\infty}$ (and also $D$) satisfies \eqref{eq:d_fixed_point}. Let $\tau_1$ be the time of the first split in $(\mathcal{Z}_t)_{t\geq 0}$, then
$$
\mathcal{Z}_t(\cdot)\overset{d}{=}{\bf 1}_{\{\tau_1>t\}}\delta_0(\cdot)+{\bf 1}_{\{\tau_1\leq t\}}\sum_{k=1}^{N}\mathcal{Z}^{(k)}_{t-\tau_1}(\cdot - z_{k,\tau_1}),
$$
and therefore
\begin{align*}
\sqrt{t}\mathcal{M}_t(\gamma^{\ast})&=\sqrt{t}e^{-\Phi(\gamma^{\ast})t}\int_{\R}e^{\gamma^{\ast}y}\mathcal{Z}_t({\rm d}y)\overset{d}{=}{\bf 1}_{\{\tau_1>t\}}\sqrt{t}e^{-\Phi(\gamma^{\ast})t}\\
&+{\bf 1}_{\{\tau_1\leq t\}}e^{-\Phi(\gamma^{\ast})\tau_1}\sum_{k=1}^{N}\sqrt{t}e^{-\Phi(\gamma^{\ast})(t-\tau_1)}A_k^{\gamma^{\ast}}\int_{\R}e^{\gamma^{\ast}y}
\mathcal{Z}^{(k)}_{t-\tau_1}({\rm d}y).
\end{align*}
Sending $t\to\infty$ yields \eqref{eq:d_fixed_point} because $\tau_1$ has the standard exponential law and is independent of $(\mathcal{Z}^{(k)}_t)_{t\geq 0}$, $k\in\N$. The proof of part (i) is complete.

\vspace{0.5mm}

\noindent
{\sc Proof of part (ii).} The claim of part (ii) can be reformulated as follows:
$$
\min_{k=1,\ldots,\mathcal{Y}_t}\left( -\gamma^{\ast}z^{\circ}_{k,t}\right)-\frac{1}{2}\log t\overset{\P}{\to}+\infty,\quad t\to\infty.
$$
Fix arbitrary $M>0$ and define a function
$$
p_{M}(t):=\P\bigg\{\min_{k=1,\ldots,\mathcal{Y}_t}\left(-\gamma^{\ast}z^{\circ}_{k,t}\right)-\frac{1}{2}\log t<M\bigg\},\quad t\geq 0.
$$
By Theorem 5.12 in \cite{Shi:2015} applied to the discrete-time branching random walk $(\mathcal{Z}_n^{(\theta)})_{n=0,1,2,\ldots}$, see also \cite{AidekonShi2010,HuShi2009}, we already know that
$$
\lim_{n\to\infty}p_{M}(n\theta)=0
$$
for every fixed $\theta>0$.  In order to finish the proof of part (ii) it remains to show that
$$
\lim_{t\to\infty,t\in\R}p_{M}(t)=0.
$$
According to the Croft-Kingman lemma, see Corollary 2 in \cite{Kingman:1963}, it is enough to check that $t\mapsto p_{M}(t)$ is right-continuous. To prove the latter statement, note that for $0\leq s\leq t$ we have
\begin{multline*}
|p_{M}(t)-p_{M}(s)|\leq \P\{\text{there are splits during } [s,t]\}\\
+\P\left\{\min_{k=1,\ldots,\mathcal{Y}_t}\left(-\gamma^{\ast}z^{\circ}_{k,t}\right)\in \Big[M+\frac{1}{2}\log s, M+\frac{1}{2}\log t\Big) \right\},
\end{multline*}
where we have used the equality $\min_{k=1,\ldots,\mathcal{Y}_t}\left(-\gamma^{\ast}z^{\circ}_{k,t}\right)=\min_{k=1,\ldots,\mathcal{Y}_s}\left(-\gamma^{\ast}z^{\circ}_{k,s}\right)$ which holds if there are no splits in $[s,t]$.
The right-hand side of the last display converges to $0$ as $s\to t+$. The proof of part (ii) is complete.
\end{proof}

\section{Proof of Theorem \ref{thm:main}}\label{sec:proof}
The key ingredient in the proof is Propostion \ref{prop: as} and the following lemma.

\begin{lem}\label{lem:3}
Assume that $(r_t)_{t\geq 0}$ is an integer-valued random process such that $r_t\overset{\P}{\to}\infty$, as $t\to\infty$. Further, suppose that for every $t\geq 0$ there is an array $(a_{k,t})_{k=1,\ldots,r_t}$ of a.s.~positive random weights such that
$$
\sum_{k=1}^{r_t}a_{k,t}^{\gamma}\overset{\P}{\to}a_{\infty}\quad\text{and}\quad \max_{k=1,\ldots,r_t}a_{k,t}\overset{\P}{\to} 0,\quad t\to\infty,
$$
for some a.s. positive random variable $a_{\infty}$ and $\gamma\in(0,2]$. Let $(X_k)_{k\in\N}$ be a sequence of independent random variables with common distribution function $F_0$ satisfying $(H_{\gamma})$ and which are independent of $(a_{k,t})_{k=1,\ldots,r_t}$ and $r_t$ for every fixed $t\geq 0$. Put
$$
S_t:=\sum_{k=1}^{r_t} a_{k,t}X_k,\quad t\geq 0.
$$
Then
$$
\lim_{t\to\infty}\E\exp(i\xi S_t)=\E \widehat{g}_{\gamma}(\xi a_{\infty}^{1/\gamma}),\quad \xi\in\R,
$$
where $\widehat{g}_{\gamma}$ is defined by \eqref{chaSta}.
\end{lem}
\begin{proof}
The proof is based on the following asymptotic expansions of the characteristic function $\phi_0$ of $X_0$, that are equivalent to the corresponding assumptions of the distribution function $F_0$, see Theorem 2.6.5 in \cite{Ibragimov+Linnik:1971}:
\begin{itemize}
\item if the case (a) of $(H_1)$ holds, then $\log \phi_0(\xi)=i m_0\xi+o(\xi)$ as $\xi\to 0$;
\item if the case (b) of $(H_1)$ holds, then $\log \phi_0(\xi)=i m_0\xi-\pi c_0^{+}|\xi|+o(\xi)$ as $\xi\to 0$;
\item if $(H_2)$ holds, then $\log \phi_0(\xi)=-\frac{\sigma_0^2}{2}\xi^2+o(\xi^2)$ as $\xi\to 0$;
\item if $(H_{\gamma})$ holds with $\gamma\in(0,1)\cup(1,2)$, then
$$
\log \phi_0(\xi)=-k_0|\xi|^{\gamma}(1-i \eta_0 \tan({\pi\gamma}/{2} )\operatorname{sign}\xi)+o(|\xi|^{\gamma}), \quad \xi\to 0.
$$
\end{itemize}
Using the above expansions the rest of the proof is standard and relies on the formula
$$
\E\exp(i\xi S_t)=\E \left(\exp\left(\sum_{k=1}^{r_t}\log \phi_0(a_{k,t}\xi)\right)\right),\quad \xi\in\R,\quad t\geq 0.
$$
We will give full details in the case (b) of $(H_1)$. The other cases can be checked similarly. From the equality
\begin{multline*}
\E\exp(i\xi S_t)=\E \Bigg(\exp\Bigg(\sum_{k=1}^{r_t}\left(\log \phi_0(a_{k,t}\xi)-i m_0 \xi a_{k,t}+\pi c_0^{+}|\xi |a_{k,t}\right)\\
+i m_0 \xi\sum_{k=1}^{r_t}a_{k,t}-\pi c_0^{+}|\xi | \sum_{k=1}^{r_t}a_{k,t}\Bigg)\Bigg),\quad \xi\in\R,\quad t\geq 0,
\end{multline*}
we see that it is enough to check that for every fixed $\xi\in\R$
\begin{equation}\label{eq:lem3}
\sum_{k=1}^{r_t}\left(\log \phi_0(a_{k,t}\xi)-i m_0 \xi a_{k,t}+\pi c_0^{+}|\xi | \sum_{k=1}^{r_t}a_{k,t}\right)\overset{\P}{\to} 0,\quad t\to\infty.
\end{equation}
Fix $\varepsilon>0$. There exists $x_0(\varepsilon)>0$ such that
$$
|\log\phi_0(x)-i m_0 x+\pi c_0^{+}|x||\leq \varepsilon |x|,\quad |x|\leq x_0(\varepsilon).
$$
Therefore, for every fixed $\varepsilon_0>0$
\begin{align*}
&\hspace{-1cm}\P\left\{\sum_{k=1}^{r_t}|\log \phi_0(a_{k,t}\xi)-i m_0 \xi a_{k,t}+\pi c_0^{+}|\xi |a_{k,t}|>\varepsilon_0\right\}\\
&\leq \P\left\{\varepsilon|\xi| \sum_{k=1}^{r_t}a_{k,t}>\varepsilon_0\right\}+\P\left\{|\xi|a_{k,t}>x_0(\varepsilon)\text { for some }k=1,\ldots,r_t\right\}\\
&=\P\left\{\varepsilon|\xi| \sum_{k=1}^{r_t}a_{k,t}>\varepsilon_0\right\}+\P\left\{|\xi|\max_{k=1,\ldots,r_t}a_{k,t}>x_0(\varepsilon)\right\}.
\end{align*}
Sending $t\to\infty$ and then $\varepsilon\to +0$ yields \eqref{eq:lem3}. The proof is complete.
\end{proof}

\begin{proof}[Proof of Theorem \ref{thm:main}]
Put $a_{k,t}:=t^{\frac{1}{2\gamma^{\ast}}}e^{z_{k,t}-t\frac{\Phi(\gamma^{\ast})}{\gamma^{\ast}}}$, $r_t:=\mathcal{Y}_t$, $a_{\infty}:=D'=\sqrt{\frac{2}{\pi (\gamma^{\ast})^2\Phi^{\prime\prime}(\gamma^{\ast})}}D_{\infty}$, $\gamma=\gamma^{\ast}$, and finally
$$
S_t=t^{\frac{1}{2\gamma^{\ast}}}e^{-\mu(\gamma^{\ast})t}\sum_{k=1}^{\mathcal{Y}_t}e^{z_{k,t}}X_k,\quad t\geq 0.
$$
From Proposition \ref{prop: as} we know that all the assumptions of Lemma \ref{lem:3} hold and therefore
$$
\lim_{t\to\infty}\E\exp\left(i\xi S_t\right)=\E \widehat{g}_{\gamma}\left(\xi \left(\sqrt{\frac{2}{\pi (\gamma^{\ast})^2\Phi^{\prime\prime}(\gamma^{\ast})}}D_{\infty}\right)^{1/\gamma}\right),\quad \xi\in\R.
$$
By Proposition \ref{prop:prob_solution}
$$
\E\exp\left(i\xi S_t\right)=\phi(t,t^{\frac{1}{2\gamma^{\ast}}}e^{-\mu(\gamma^{\ast})t}\xi)
$$
which proves convergence. The fact that $D_{\infty}$ satisfies \eqref{eq:d_fixed_point} has already been proved above. The proof of Theorem \ref{thm:main} is complete.
\end{proof}

\section*{Acknowledgment} We thank two anonymous referees for the detailed and useful reports containing numerous remarks and suggestions.

\end{document}